\documentclass{artjlt}

\usepackage{amsmath,amsfonts,amssymb}
\allowdisplaybreaks[1]

\newtheorem{thm}{Theorem}[section]

\newtheorem{cor}[thm]{Corollary}
\newtheorem{lem}[thm]{Lemma}
\newtheorem{prop}[thm]{Proposition}
\newtheorem{conj}[thm]{Conjecture}

\theoremstyle{definition}

\numberwithin{equation}{section}

\def\O{\Omega}

\newcommand{\R}{\mathbb R}

\def \l{\lambda}
\def \a{\alpha}

\renewcommand{\b}{\beta}

\def\g{\mathfrak g}

\newcommand{\h}{\mathfrak h}

\renewcommand{\H}{\mathcal H}

\title{Admissibility For Monomial Representations of Exponential Lie Groups}                                     %<-------------------
\author{Bradley Currey and Vignon Oussa}                 %<-------------------
\lastname{Currey and Oussa}  %<-------------------

    %<-------------------

         %<-------------------

\address{%
Bradley Currey\\               %<-------------------
Department of Mathematics and Computer Science\\
220 N. Grand Blvd.\\
St.Louis, MO 63103\\            %<-------------------
curreybn@slu.edu               %<-------------------
}
\address{%
Vignon Oussa\\               %<-------------------
Department of Mathematics and Computer Science\\
220 N. Grand Blvd.
\\St.Louis, MO 63103\\            %<-------------------
voussa@slu.edu               %<-------------------
}
\begin{document}

\maketitle

\begin{abstract}
 Let $G$ be a simply connected
exponential solvable Lie group, $H$ a closed connected subgroup, and let $\tau$ be a representation of $G$ induced from a unitary character $\chi_f$ of $H$. The spectrum of $\tau$ corresponds via the orbit method to the set $G\cdot A_\tau / G$ of coadjoint orbits that meet the spectral variety $A_\tau = f + \h^\perp$. We prove that the spectral measure of $\tau $ is absolutely continuous with respect to the Plancherel measure if and only if $H$ acts freely on some point of $A_\tau$. As a corollary we show that if $G$ is nonunimodular, then $\tau$ has admissible vectors if and only if the preceding orbital condition holds.  %<------------------- 
\end{abstract}

\section{Introduction}

At the intersection of abstract harmonic analysis and wavelet theory lies the fundamental notion of admissibility. Given a unitary representation $\tau$ of a locally compact topological group $G$, a vector $\psi \in \H_\tau$ is admissible if the mapping $\phi \mapsto \langle \phi,\tau(\cdot)\psi\rangle$ is an isometry of $\H_\tau$ into $L^2(G)$. Which representations have admissible vectors? This classical question is answered in a variety of contexts and for various classes of representations, for example when $G$ is type I and $\tau$ is irreducible \cite{DM}, or when $\tau$ is the left regular representation of $G$ \cite{F02}. The monograph  \cite{Fu05}, in addition to containing numerous other references for admissibility, describes the relation between this question and Plancherel theory.

In this paper we consider the following class of representations. Let $G$ be an exponential solvable Lie group, and let $\tau$ be the unitary representation of $G$ induced from a unitary character of $H$. A description of the irreducible decomposition of $\tau$ is given in terms of the coadjoint orbit picture in \cite{FUJ}. On the other hand an explicit Plancherel formula for $G$ is given in \cite{C2} using coadjoint orbit parameters. Using these results, we give a simple necessary and sufficient condition that $\tau$ be a subrepresentation of the left regular representation of $G$ in terms of the orbit picture. Specifically, let $\tau$ be induced from the character $\chi$ of $H$, let $f$ be the linear functional on $\h$ corresponding to $\chi$, so that $\chi(\exp Y) = e^{if(Y)}$. Let $A_\tau$ be the real affine variety of all $\ell \in \g^*$  whose restriction to $\h$ is $f$. Then $A_\tau$ is an $\text{Ad}^*H$-space, and we show (Theorem \ref{measurethm}) that $\tau$ is contained in the left regular representation if and only if $H$ acts freely on some $\ell \in A_\tau$ (and hence on a Zariski open subset of $ A_\tau$.) Combining this result with the methods of \cite{Fu05}, it follows that if $G$ is nonunimodular, then the preceding condition is both necessary and sufficient in order that $\tau$ have admissible vectors (Corollary \ref{admissible}). If $G$ is unimodular, then the situation for admissibility is still murky.

%---------------------------------------------------------------------------------------------------------------------------------------------------
\section{Preliminaries}\label{prel}

Let $G$ be a connected, simply connected exponential solvable Lie group
with Lie algebra $\g$. Given $s\in G$, $Z\in\mathfrak g$,
and $\ell\in\g^*$, we denote both the adjoint and coadjoint actions
multiplicatively: Ad$(s)Z=s\cdot Z$ and $\text{Ad}^*(s)\ell= s\cdot
\ell$. Given $\ell \in \g^*$, let $G(\ell)$ be the stabilizer of $\ell$ in $G$; then $G$ is connected and its Lie algebra is $\g(\ell) = \{ X \in \g : \ell [X,Z] = 0 \text{ holds for all } Z \in \g\}$.

For the remainder of this paper, we fix a closed connected subgroup $H$ of $G$ with Lie algebra $\h$, a unitary character $\chi$ of $H$, and a {\it monomial} representation $\tau = \text{ind}_H^G(\chi)$. Let $f \in \h^*$ satisfy $\chi(\exp Y) = e^{if(Y)}$ so that  $[\h,\h] \subset \ker f$; we also use the notation $\tau = \tau(f,\h)$. In our notation, we will identify representations that are unitarily equivalent. Let $\hat G$ denote the Borel space of equivalence classes of irreducible representations of $G$. 

Let $\tau$ be a monomial representation. Since $G$ is type I, we have a unique measure class on $\hat G$ such that 
$$
\tau = \int_{\hat G}^\oplus m_\tau(\pi) \pi d\nu(\pi).
$$
In particular, the Plancherel measure class $\mu$ is the measure class determined by the regular representation $L$. Since the multiplicity $m_L(\pi) = \infty$ $\mu$-a.e., then $\tau$ is a subrepresentation of $L$ if and only if $\nu$ is absolutely continuous with respect to $ \mu$.

The measure class $\nu$ is described on $\g^*/G$ as follows: for $\tau = \tau(f,\h)$, we put $A_\tau = f + \h^\perp = \{ \ell \in \g^* : \ell|_{\h} = f\}$. Let $\xi$ be the canonical Lebesgue measure class on $A_\tau$ extended to $\g^*$: for any Borel subset $B$ of $\g^*$, $\xi(B) = \xi(B\cap A_\tau)$. By \cite{FUJ}, $\nu$ is the pushforward of $\xi$ to $\g^*/G$.
%, and $m_\tau(\pi)$ is $\nu$-a.e. the number of $H$-orbits in $\eO_\pi \cap A$. 
Though $\xi$ is singular with respect to the Lebesgue measure class on $\g^*$ (unless $H$ is the trivial subgroup), its pushforward $\nu$ may be absolutely continuous with respect to the Plancherel measure. In the next section, we will determine when this is the case.

\section{Absolute continuity of the spectral measure.}\label{measure}

Consider the action of $H$ on $\g^*$ by the restriction of the coadjoint action. Given $\ell \in \g^*$, the Lie algebra of the stabilizer of $\ell$ in $H$ is $\h(\ell) = \h \cap \g(\ell)$. If $s \in G$, then one observes that $\h(s\cdot\ell) = s \cdot \h(\ell)$. If $U(d) = \{ \ell \in \g^* : \dim H\cdot \ell = d\}$, then the preceding observation shows that each $U(d)$ is $G$-invariant. Put 
$$d_\tau = \max \{ d : U(d) \cap A_\tau \ne \emptyset\}
$$ 
and $V  := U(d_\tau)\cap A_\tau$. We claim that $V$ is a Zariski open subset of $A_\tau$. 
Indeed, fix a basis $\left\{ Y_{1},Y_{2},\cdots,Y_{m}\right\}$ for $\h$ and a basis $\left\{ Z_{1},Z_{2},\cdots,Z_{n}\right\}$ for $\g$.  For $\ell \in \g^*$ let $M(\ell)$ be the $m \times n$ matrix
$$M(\ell) = \left[\begin{matrix} \ell [Y_{1}, Z_1] &\ell [Y_{1}, Z_2] &\cdots  &\ell [Y_{1}, Z_n]  \\
\ell [Y_{2}, Z_1]  & \ell [Y_{2}, Z_2] &\cdots &\ell [Y_{2}, Z_n]  \\
\vdots &\vdots &\cdots  &\vdots \\
\ell [Y_{m}, Z_1]  & \ell [Y_{m}, Z_2] &\cdots & \ell [Y_{m}, Z_n]
 \end{matrix}\right].$$
%and put $r  $\rho = \max \{ \text{rank}(M(\ell)) : \ell \in A\}$. Note that the set $\{ \ell \in A : \text{rank}(M(\ell)) = \rho\}$ is a Zariski-open subset of $A$.

\begin{lem}\label{hdim} Let $\ell \in \g^*$. Then $\dim H\cdot \ell = \text{\rm rank}\ M(\ell)$. Hence $V$ is Zariski-open in $A_\tau$. 

\end{lem}

\begin{proof} We have $\dim H\cdot \ell = \h / \h(\ell)$, where $\h(\ell)$ is the Lie algebra of the stabilizer of $\ell$ in $H$. Now $\h(\ell) = \g^\ell \cap \h = \{ Y \in \h : \ell [Y,X] = 0 \text{ for all } X \in \g\}$. It is easily seen that $\dim \h / \h(\ell) =  \text{rank}M(\ell)$.\qedhere \end{proof}

We are especially interested in the case where $d_\tau = \dim H$.

\begin{cor} Suppose that there is some $\ell \in A_\tau$ such that $H$ acts freely on $\ell$. Then $H$ acts freely on a Zariski open subset of $A_\tau$.

\end{cor}

Define the smooth map $\phi :G\times A_\tau\rightarrow \mathfrak{g}^{\ast }$ by
$$\phi \left( s,\ell\right)
=s\cdot \ell.
$$
Choose a basis $\{Z_1, Z_2, \dots , Z_n\}$ for $\g$ with the following properties.

\begin{itemize}

\item For each $j, 1 \le j< n$, set $\g_j = \text{span} \{Z_1, Z_2, \dots , Z_j\}$. If $\g_j$ is   not an ideal in $\g$, then $\g_{j+1}$ and $\g_{j-1}$ are ideals. 

\item If $\g_{j}$ is not an ideal in $\g$, then the module $\g_{j+1} / \g_{j-1}$ is not $\R$-split. 

\end{itemize}

Then $(t_1, t_2, \dots , t_n) \mapsto \exp t_{1}Z_{1}\cdots \exp t_{n}Z_{n}$ is a global diffeomorphism of $\R^n$ onto $G$. %We define a global chart for $G\times A_\tau$ as follows. 
Recall the basis $Y_1, Y_2, \dots , Y_m$ of $\h$ and put $f_j = f(Y_j), 1 \le j \le m$. Choose $X_1, X_2, \dots, X_{n-m}$ so that $Y_1, Y_2, \dots , Y_m, X_1, X_2, \dots , X_{n-m}$ is an ordered basis of $\g$. Let $\b$ be the natural global chart for $\g^*$ determined by the ordered dual basis $Y_1^*, Y_2^*, \dots Y_m^*, X_1^*, X_2^*, \dots , X_{n-m}^*$. 
%Let $X_1, \cdots X_{n-m}$ be a basis for $\g\mod\h$; 
Now define $\a : \R^n \times \R^{n-m} \rightarrow G\times A_\tau$ by
$$
\a(t,x)  = \Bigl(\exp t_{1}Z_{1}\cdots \exp t_{n}Z_{n},\  f_{1}Y_{1}^{\ast }+\cdots +f_{m}Y_{m}^{\ast }+x_{1}X_{1}^{\ast }+\cdots
+x_{n-m}X_{n-m}^{\ast }\Bigr)
$$
%Let $\b$ be the natural global chart for $\g^*$ determined by the ordered dual basis $Y_1^*, Y_2^*, \dots Y_m^*, X_1^*, X_2^*, \dots , X_{n-m}^*$. 
Note that for $t = 0$, $\a(0,\cdot) : \R^{n-m} \rightarrow \{e\} \times A_\tau \simeq A_\tau$ defines a global diffeomorphism from $\R^{n-m}$ onto $A_\tau$ (here $e$ is the identity in $H$). Moreover $\a^{-1}$ is a global chart for $G \times A_\tau$ and  $ \b \circ \phi \circ \a$ is a coordinatization of the map $\phi$. For simplicity of notation we set $\tilde\phi = \phi \circ \a$; observe that the coordinate functions for the map $ \b \circ \phi \circ \a$ are given by
$$
( \b \circ \phi \circ \a)_j(t,x) = \tilde\phi(t,x)(Y_j) = \phi\bigl(\a(t,x)\bigr)(Y_j), \ \ 1 \le j \le m,
$$
and
$$
 (\b \circ \phi \circ \a)_j(t,x) = \tilde\phi(t,x)(X_{j-m}) = \phi\bigl(\a(t,x)\bigr)(X_{j-m}), \ \ m+1 \le j \le n.
$$

\vskip 0.5cm

\begin{lem}\label{rank}
For each $\ell \in A_\tau$,
$$
\text{\rm rank}\ d\phi(e,\ell) = \dim H\cdot \ell + n -m = n - \dim H(\ell).
$$
%Let $\ell \in U$. Then \text{rank}$(d\phi(0,\ell))=dim(\n).$
\end{lem}
\begin{proof} Let $\ell \in A_\tau$ and $x$ the corresponding point in $ \R^{n-m}$ such that $\a(0,x) = (e,\ell)$.
%Let $X_1, \cdots X_{n-k}$ be a basis for $\g\mod \h$. We define coordinate functions for  $\phi$ as follows: $$\phi _{1}(t,l)=\phi \left( t,l\right) \left( Y_{1}\right) ,\cdots ,\phi_{k}(t,l)=\phi \left( t,l\right) \left( Y_{k}\right) $$ and $$\phi_{k+1}(t,l)=\phi \left( t,l\right) \left( X_{1}\right) ,\cdots ,\phi_{n}(t,l)=\phi \left( t,l\right) \left( X_{n-k}\right) .$$
%Each $i$-th coordinate function simply represents the restriction of the coadjoint orbit in the direction of $Z_{i}$ .
%We write
%$$
%\ell=f_{1}Y_{1}^{\ast }+\cdots +f_{k}Y_{k}^{\ast }+x_{1}X_{1}^{\ast }+\cdots
%+x_{n-k}X_{n-k}^{\ast }$$
%and
%$$\tilde\phi(t,x) = \phi\left(t, f_{1}Y_{1}^{\ast }+\cdots +f_{k}Y_{k}^{\ast }+x_{1}X_{1}^{\ast }+x_{n-k}X_{n-k}^{\ast }\right)$$
We compute the Jacobian matrix $J_\phi$ of the coordinatization $\b \circ \phi \circ \a$ of $\phi$ at $(0,x)$:

\vskip 0.2cm

\begin{eqnarray*}
J_{\phi }\left( 0,x\right)  &=&\left[
\begin{array}{cccccc}
\frac{\partial \tilde\phi\left( 0,x \right)(Y_1) }{\partial t_{1}} & \cdots  & \frac{\partial \tilde\phi\left( 0,x \right)(Y_1) }{\partial t_{n}} &
\frac{\partial \tilde\phi\left( 0,x \right) }{\partial x_{1}} & \cdots  &
\frac{\partial \tilde\phi\left( 0,x \right)(Y_1) }{\partial x_{n-m}} \\
\vdots  & \ddots  & \vdots & \vdots & \ddots  & \vdots  \\
\frac{\partial \tilde\phi\left( 0,x \right)(Y_m) }{\partial t_{1}} & \cdots  & \frac{\partial \tilde\phi\left( 0,x \right)(Y_m) }{\partial t_{n}} &
\frac{\partial \tilde\phi\left( 0,x \right) }{\partial x_{1}} & \cdots  &
\frac{\partial \tilde\phi\left( 0,x \right)(Y_m) }{\partial x_{n-m}} \\
\frac{\partial \tilde\phi\left( 0,x \right)(X_1) }{\partial t_{1}} & \cdots  & \frac{\partial \tilde\phi\left( 0,x \right)(X_1) }{\partial t_{n}} &
\frac{\partial \tilde\phi\left( 0,x \right) }{\partial x_{1}} & \cdots  &
\frac{\partial \tilde\phi\left( 0,x \right)(X_1) }{\partial x_{n-m}} \\
\vdots  & \ddots  & \vdots & \vdots  & \ddots  & \vdots  \\
\frac{\partial \tilde\phi\left( 0,x \right)(X_{n-m}) }{\partial t_{1}} & \cdots  & \frac{\partial \tilde\phi\left( 0,x \right)(X_{n-m}) }{\partial t_{n}} &
\frac{\partial \tilde\phi\left( 0,x \right) }{\partial x_{1}} & \cdots  &
\frac{\partial \tilde\phi\left( 0,x \right)(X_{n-m}) }{\partial x_{n-m}}
\end{array}\right]
\end{eqnarray*}

\vskip 0.5cm
\noindent
Now for each $1 \le j \le m$, we have
$$
\frac{\partial \tilde\phi\left( 0,x \right)(Y_j) }{\partial t_{k}}  =
\left.\frac{d}{du}\right|_{u = 0} \ell \left( Y_j + u [Y_j,Z_k] + \frac{u^2}{2!} [[Y_j,Z_k],Z_k] + \cdots \right) = \ell[Y_j,Z_k]
$$
holds for each $1 \le k \le n$, while
$$
\frac{\partial \tilde\phi\left( 0,x\right)(Y_j) }{\partial x_{r}}  = 0,  \ \ 1 \le r \le n-m
$$
since $\ell \mapsto \tilde\phi(0,x)(Y_j)$ is constant. On the other hand,
$$
\frac{\partial \tilde\phi\left( 0,x \right)(X_s) }{\partial x_{r}}  =  \delta_{rs}.
$$
Hence the differential $d\phi(e,\ell)$ is given by the matrix
$$
J_{\phi }\left( 0,x\right) = \left[\begin{matrix} M(\ell) \ \ & \mathbf 0 \\ \ast & I \end{matrix} \right]
$$
where $\mathbf 0$ denotes the $m \times (n-m)$ zero matrix and $I$ denotes the $(n-m) \times (n-m)$ identity matrix.
Now by Lemma \ref{hdim}, we have rank of $M(\ell) = \dim H \cdot \ell$, and the result follows.\qedhere \end{proof}

We can now state the main result.

\begin{thm}\label{measurethm} Let $H$ be a closed connected subgroup of $G$ with Lie algebra $\h$ and let $\tau = \tau(f,\h)$. If $H$ acts freely on some point of $A_\tau$
%$H(\ell) = \{1\}$, 
then $\nu$ is absolutely continuous with respect to $\mu$. Otherwise, $\nu$ is singular with respect to $\mu$.
\end{thm}

\begin{proof} Suppose that there is $\ell \in A_\tau$ such that $H(\ell) = \{e\}$. By Lemmas \ref{hdim} and  \ref{rank}, $V = \{ \ell \in A_\tau  : \text{rank }d\phi(e,\ell)= n\}$ is a non-empty Zariski open subset of $A_\tau$. For each $\ell \in V$ there exists a rectangular open neighbourhood $J_\ell\times V_\ell$ of $(e,\ell)$ such that the restriction of $\phi$ to $J_\ell\times V_\ell$ is a submersion, and hence that $W_{\ell} = \phi(J_\ell\times V_\ell)$ is open. Now
$$W=\bigcup_{\ell \in V}W_{\ell}
$$
is open and satisfies $V \subset W \subset G\cdot V$.  Hence $G\cdot V / G = G\cdot W / G$ is open in $\g^* /G$.

Next we invoke results concerning the stratification and parametrization of coadjoint orbits  \cite{C0, C2} : there is a $G$ invariant Zariski-open subset $\O$ of $\g^*$, such that $\O / G$ has the structure of a smooth manifold (with underlying quotient topology) and such that the quotient mapping $ \sigma : \O \rightarrow \O / G$ is real analytic.  Now since $\O$ is dense in $\g^*$, then $\O \cap W \ne \emptyset$. Since $W \subset G\cdot V$ and $\O$ is $G$-invariant, then  $\O \cap V$ is a non-empty Zariski-open subset of $A_\tau$, and the Lebesgue measure $\xi$ on $A_\tau$ is supported on $\O \cap V$. Since $G \cdot( \O \cap V) $ is included in the $G$-invariant set $ \O \cap U(d_\tau)$, then $G \cdot (\O\cap V)$ is disjoint from the set $ G\cdot(A_\tau \setminus (\O \cap V))$, and hence $\nu$ is supported on $G \cdot (\O\cap V) / G$. Put $\a = \sigma|_{\O \cap V}$; since $\sigma$ is real analytic, then so is $\a$. Moreover, $G\cdot (\O \cap V) /G$ is open in $\g^* /G$ and $\nu = \a_\ast\xi$. Since $\a$ is real analytic on $A_\tau$, its set of singular points in $A_\tau$ has $\xi$-measure zero, and $\a$ is a submersion on the set of regular points in $A_\tau$. Since the pushforward of Lebesgue measure by a submersion is absolutely continuous with respect to Lebesgue measure, then $\nu$ is absolutely continuous with respect to the Lebesgue measure class on $G\cdot ( \O\cap V) /G$. Since Plancherel measure $\mu$ on $\O / G$ belongs to the Lebesgue measure class on $\O / G$ \cite{C2} and $G\cdot ( \O\cap V) /G$ is an open subset of $\O / G$, then $\nu$ is absolutely continuous with respect to $ \mu$.

Now suppose that for all $\ell \in A_\tau$, $H(\ell)$ is non-trivial. Then for all points $\ell \in A_\tau$, the rank of $\phi$ at $(e,\ell)$ is less than $n$. It follows that the Lebesgue measure of $G\cdot V$ is zero, and hence $\mu(G\cdot V / G) = 0$. But since $V$ is a Zariski-open subset of $A_\tau$, then the measure $\nu$ is supported on $G\cdot V / G$, and hence $\nu$ is singular with respect to $\mu$.\qedhere \end{proof}

We now turn to the question of admissibility. Let $\pi$ be any representation of $G$ acting in $\mathcal{H}_{\pi}$. For $\eta \in \H$ define $W_\eta : \H_\pi \rightarrow C(G)$ by $W_\eta(f) = \langle f, \pi(\cdot) \eta\rangle$. The vector $\eta$ is said to be admissible (or a continuous wavelet) if $W_\eta$ is an isometry of $\mathcal{H}$ into $L^2(G)$. In this case, $W_\eta $ intertwines the representation $\pi$ with the left regular representation $L$ of $G$, so that $\H = W_\eta(\H_\pi)$ is a closed left invariant subspace of $L^2(G)$ and $\pi$ is equivalent with $L$ acting in $\H$. 

Let $\mathcal H$ be a closed left invariant subspace of $L^2(G)$, and let $P : L^2(G) \rightarrow \mathcal H$ be the orthogonal projection onto $\mathcal H$. Then there is a unique (up to $\mu$-a.e. equality) measurable field $\{\hat P_\l\}_{\l\in \hat G}$ of orthogonal projections where $\hat P_\l$ is defined on $\mathcal L_\l$, and so that 
 $$
 \widehat{(P\phi)}(\l) = \hat\phi(\l) \hat P_\l
 $$
 holds for $\mu$-a.e.  $\l \in \hat G$. Set $m_\mathcal H(\l) = \text{rank}(\hat P_\l)$. 
%If $\pi$ has an admissible vector $\eta$, and if $\H_0$ is a $\pi$-invariant subspace of $\H$, then the projection $\eta_0$ of $\eta$ in $\H_0$ is admissible for the corresponding subrepresentation of $\pi$. Finally, since $G$ is connected, it follows from results in \cite{F02} that $\L$ is admissible if and only if $G$ is nonunimodular.
We recall  \cite[Theorem 4.22]{Fu05}.
 
 \begin{prop} \label{admissible} Let $\mathcal H$ be a closed left invariant subspace of $L^2(G)$. If $G$ is nonunimodular, then $\mathcal H$ has an admissible vector. If $G$ is unimodular, then $\H$ has an admissible vector if and only if $m_\mathcal H$ is integrable over $\hat G$ with respect to the Plancherel measure $\mu$.  
 
% \vspace{.1in}
% \noindent
% (a)  $\mathcal H$ has an admissible vector. 
 
% \vspace{.1in}
% \noindent
% (b) There is a left invariant subspace $\mathcal K$ of $L^2(G)$ and  $\eta \in\mathcal K$ such that $\phi \mapsto \phi * \eta^*$ is an isometric isomorphism of $\mathcal K$ onto $\mathcal H$. 
 
 %  \vspace{.1in}
 %\noindent
% (c) There is a unique self-adjoint  convolution idempotent $S\in \mathcal H$ such that $\mathcal H = L^2(G) * S$. 
 
%  \vspace{.1in}
% \noindent
% (d) The function $m_\mathcal H$ is integrable over $\hat G$ with respect to Plancherel measure $\mu$.  

 \end{prop}

In light of the preceding and Theorem \ref{measurethm}, the following is immediate. 

\begin{cor} \label{admissible} Suppose that $G$ is nonunimodular. Then $\tau$ has an admissible vector if and only if $H$ acts freely on some $\ell \in A_\tau$. 
\end{cor}

Suppose that $G$ is unimodular. Though it is clear that the condition  that $H$ acts freely on points of $A_\tau$ is still necessary for admissibility, examples indicate that the multiplicity function is never integrable, and hence that $\tau$ never has admissible vectors in this case. Thus we make the following. 

\begin{conj} A monomial representation of a unimodular exponential solvable Lie group $G$ never has admissible vectors. 

\end{conj}
A resolution of this conjecture would require a more precise understanding of the image of the set $G\cdot(\O \cap V) / G$ in $\g^* / G$.

% Your bilbigraphy           %<-------------------

\end{document}